\documentclass{article}

\usepackage{amsmath}
\usepackage{amssymb}
\usepackage{amsthm}

\newtheorem{theorem}{Theorem}[section]
\newtheorem{proposition}[theorem]{Proposition}
\newtheorem{lemma}[theorem]{Lemma}

\newcommand{\CC}{\mathbb{C}}

\newcommand{\ZZ}{\mathbb{Z}}
\newcommand{\RR}{\mathbb{R}}

\newcommand{\sgn}{\mathop{\mathrm{sgn}}}
\newcommand{\up}{u_\mathrm{p}}
\newcommand{\ue}{u_\mathrm{e}}
\newcommand{\be}{\beta_\mathrm{e}}
\newcommand{\tup}{\tilde u_\mathrm{p}}
\newcommand{\upa}{u_\mathrm{pa}}
\newcommand{\vp}{v_\mathrm{p}}
\newcommand{\uodd}{u_\mathrm{o}}
\newcommand{\loc}{\mathrm{loc}}
\newcommand{\odd}{\mathrm{odd}}
\renewcommand{\o}{\mathrm{o}}
\newcommand{\e}{\mathrm{e}}
\newcommand{\abs}[1]{\left\lvert #1\right\rvert}
\newcommand{\norm}[1]{\left\lVert #1\right\rVert}

\renewcommand{\lefteqn}[2]{\makebox[#1\linewidth][l]{$\displaystyle #2$}}
\newcommand{\KK}{\frac{c^2k_0^2-2}{c^2k_0^2-k_0}}
\newcommand{\K}{B}
\begin{document}

\title{Travelling waves for a Frenkel-Kontorova chain}

\author{Boris Buffoni, Hartmut Schwetlick and Johannes Zimmer}
\date{July 2015}

\maketitle

\begin{abstract}
  In this article, the Frenkel-Kontorova model for dislocation dynamics is considered, where the on-site potential
  consists of quadratic wells joined by small arcs, which can be spinodal (concave) as commonly assumed in physics. The
  existence of heteroclinic waves ---making a transition from one well of the on-site potential to another--- is proved
  by means of a Schauder fixed point argument. The setting developed here is general enough to treat such a
  Frenkel-Kontorova chain with smooth ($C^2$) on-site potential. It is shown that the method can also establish the
  existence of two-transition waves for a piecewise quadratic on-site potential.

\end{abstract}

\section{Introduction}
\label{sec:Introduction}

In this article, we study the advance-delay difference-differential equation
\begin{equation}
  \label{eq:introeq}
  c^2u''-\Delta_D u+\alpha u-\alpha \psi'(u) =0
\end{equation}
on $\RR$, where $\Delta_D$ is the discrete Laplacian,
\begin{equation*}
  \Delta_D u(x):=u(x+1)-2u(x)+u(x-1);
\end{equation*}
the derivative $g'(u)$ of the on-site potential
\begin{equation*}
  g(u)=\frac 1 2\alpha  u^2-\alpha \psi(u)
\end{equation*}
will be discussed in detail below, since it presents the main challenge of this problem by being non-monotone.

In a nutshell, the main result of this article is that a solution to~\eqref{eq:introeq} exists for suitable choices of
parameters, for nonlinearities which are suitable mollified versions of the sign function,
$\alpha\psi'(u) \approx \alpha \sgn(u)$.

Mathematically, this equation combines a number of difficulties. It combines a differential operator (the second
derivative) with a difference operator ($\Delta_D$). See, e.g.,~\cite{Hale1977a} for the subject of such functional
equations. Here the equation is looking `forward', $u(x+1)$, and `backward', $u(x-1)$. The theory of such advance-delay
equations is still not very well developed, though there are very remarkable results, employing tools ranging from
variational techniques to centre manifold/normal form analysis, for
example~\cite{Friesecke1994a,Iooss2000a,Calleja2009a}. The non-monotonicity of $g'$ finally is the core difficulty of
the problem.

Physically,~\eqref{eq:introeq} is the travelling wave equation for the so-called Frenkel-Kontorova model of dislocation
dynamics~\cite{Frenkel1939a}. There, the model proposed is
\begin{equation}
  \label{eq:fk-orig}
  m u_k'' = \beta(u_{k+1} - 2 u_k + u_{k-1}) - 2\pi \frac \alpha \gamma \sin\left(\frac{2\pi}{\gamma} u_k\right) 
\end{equation}
with some constants $\alpha$, $\beta$, $\gamma$, describing the displacement $u_k$ of at atom $k\in\ZZ$ in a
one-dimensional chain; the nonlinearity is the derivative of an on-site potential describing the interaction with atoms
above and below the chain of atoms considered. The periodicity of the nonlinearity thus reflects the periodic nature of
a crystalline lattice. The Frenkel-Kontorova chain is a fundamental model of dislocation dynamics, describing how an
imperfection (dislocation) travels through a crystalline lattice; see in particular the survey~\cite{Braun1998a}. The
simplest motion that may exist is that of a travelling wave, $u_j(t) = u(j - ct)$ with wave speed $c$. This
\emph{ansatz} transforms~\eqref{eq:fk-orig}, after rescaling, into~\eqref{eq:introeq}, with sinusoidal on-site
potential $g$.

We study the situation where this potential is piecewise quadratic, with small concave parts smoothing out the cusp at
the meeting point of two parabola. For piecewise quadratic on-site potentials, there is a long history of formal
solutions, going back at least to Atkinson and Cabrera~\cite{Atkinson1965a}.  It has been pointed out that formal
calculations often depend on the validity of a sign condition (which will be encountered here as
well)~\cite{Earmme1974a,Kresse2003a}.

There are few rigorous results for nonconvex interaction potentials available, in particular for heteroclinic solutions
as we will study. A very remarkable existence result for such solutions is that of Iooss and
Kirchg\"assner~\cite{Iooss2000a}; there a general theory for small solutions is developed. Here we are interested in
(large) heteroclinic solutions that stay asymptotically for $x\to-\infty$ in one well of a nonconvex on-site potential
$g$ and for $x\to\infty$ in another well. For the particular choice $\alpha\psi'(u) = \alpha \sgn(u)$, the existence of
such travelling waves has been established for suitable parameters with an argument based on Fourier
estimates~\cite{Kreiner2011a}. Here we show that this result holds true in greater generality, in particular for
on-site potentials where the concave part is not degenerate as it is assumed in~\cite{Kreiner2011a}. We work in a
nonlinear setting where the Fourier methods of~\cite{Kreiner2011a} are not applicable.

The existence of heteroclinic travelling waves for the Frenkel-Kontorova problem~\eqref{eq:fk-orig} has been open since
1939 (for coherent spatially localised temporally periodic solutions, existence was established in the seminal paper by
MacKay and Aubry~\cite{MacKay1994a}; see also~\cite{Peyrard1984a}). We are presently unable to answer this question for
the sinusoidal on-site potential, since we use the explicit knowledge of wave trains in harmonic chains. One
interpretation of our result is that it shows that wave trains in one well of $g$ can be joined to another train in
another well, and this transition signifies a moving dislocation. We can establish this result for a class of smooth
potentials which have harmonic wells and small spinodal (concave) regions. Since the potentials we consider are
structurally very similar to the original sinusoidal on-site potential, one would expect that existence holds for that
potential as well, under similar choices of the parameters made. Yet a proof of this conjecture seems far from
straightforward.

We remark that for the Fermi-Pasta Ulam chain with smooth nonconvex interaction potential, a different approach has
been employed to prove the existence of heteroclinic waves for cases where the potential has a small spinodal (concave)
region~\cite{Herrmann2013a}. As the method used here, the approach relies on a perturbation argument, but then proceeds
differently by relying on the Banach fixed point theorem, following a careful analysis of an integral equation
describing the travelling wave equation.

The framework developed in the present article is relatively flexible and allows potentially the analysis of a range of
problems in the setting of (at least) the Frenkel-Kontorova chain. To give an example, we study in
Section~\ref{sec:Two-transition-solut} the problem with a piecewise quadratic on-site potential, $\psi'(u) = \sgn(u)$,
and establish what is to our knowledge the first proof of solutions exhibiting two transitions between the wells of the
on-site potential. It can be regarded as a simplified version of the shadowing lemma~\cite{Angenent1987a}.

\section{Setup and main result}
\label{sec:Setup-main-result}

The central argument we are going to employ is a Schauder fixed point theorem. This is possibly surprising, as
equation~\eqref{eq:introeq} is defined on the whole real line and therefore there is \emph{a priori} no reason to
expect compactness properties for~\eqref{eq:introeq}. We now sketch the setting in which the Schauder theorem applies.

We start by considering the linear part of~\eqref{eq:introeq}. The linear operator
\begin{equation}
  \label{eq: L}
  u\rightarrow Lu=c^2u''-\Delta_D u+\alpha u
\end{equation} 
has in Fourier space the representation
\begin{equation}
  \label{eq:disp}
  -c^2\zeta^2+2(1-\cos \zeta)+\alpha=
  -c^2\zeta^2+4\sin^2(\zeta/2)+\alpha
  =:D(\zeta), 
\end{equation}
where $D$ is the \emph{dispersion function}. Obviously, for the sound speed, $c=1$, the dispersion relation $D$ has
exactly two nonzero roots $\pm k_0$, where
\begin{equation}
  \label{eq:k0}
  k_0 := \frac \pi 2 
\end{equation}
if
\begin{equation}
  \label{eq:alpha}
  \alpha = c^2\left(\frac\pi 2\right)^2 - 2, 
\end{equation}
and furthermore $D'(\zeta)=-2c^2\zeta+2\sin\zeta$ vanishes only at $\zeta=0$. We will work in a parameter regime where
$c$ is marginally subsonic; we keep $k_0$ fixed by~\eqref{eq:k0} and $\alpha$ given by~\eqref{eq:alpha}. Then $c$ is
the only free parameter in the dispersion relation. Since we seek to finds heteroclinic solutions, we will focus on
subsonic waves, that is, $c\leq 1$.

By continuity, the dispersion function will have exactly two roots near $\pm k_0$ for `near sonic' subsonic $c$.

Our main theorem can be considered as perturbation result of~\cite{Kreiner2011a}, where the special case
$\psi'(u) = \sgn(u)$ is considered. We sketch the situation for this degenerate potential briefly.  For
$\abs{\lambda}<1$ and $\theta\in[0,2\pi)$, trivially $1+\lambda\sin(k_0\cdot+\theta)$ is a solution
to~\eqref{eq:introeq} on $[1,\infty)$ and $-1+\lambda\sin(k_0\cdot-\theta)$ is a solution on $(-\infty,-1]$. The
question is whether these two solution segments can be glued together to form a heteroclinic solution, traversing from
one well of the on-site potential $g$ to another.

The answer is affirmative for the degenerate potential discussed in this paragraph, as shown in~\cite{Kreiner2011a}
(recalled in Theorem~\ref{theo:KZ} below).  This solution $u\in H^2_\loc(\RR)$ is odd, $u(x) = - u(-x)$, and
heteroclinic in the sense that
\begin{equation*}
  \lim_{x\rightarrow \pm\infty}\left[u(x)\mp 1-\lambda\sin(k_0x\pm \theta)\right]=0
\end{equation*}
for some $\lambda$ and $\theta$, and $\alpha$ given by~\eqref{eq:alpha}. This solution is well approximated by the
explicit function
\begin{equation}
  \label{eq:KZ-profile}
  \upa(x)  := \sgn(x) 
  \left[A \left(1 - e^{-\beta\abs{z}}\right) + B \left(1 - \cos\left(k_{0} z\right) \right)\right] ,
\end{equation}
with
\begin{equation}
  \label{AandB}
  A =\frac{c^{2} k_{0}^{2} - \alpha}{c^{2} \left(\beta^{2}+k_{0}^{2}\right)}
  \qquad\text{and}\qquad
  B= \frac{\alpha + \beta^{2} c^{2}}{c^{2} \left(\beta^{2}+k_{0}^{2}\right)}
\end{equation}
and 
 \begin{equation*}
  \beta^{2} = 
  \frac{\alpha}{c^{2}}\cdot 
  \frac{k_{0}\sin\left(k_{0}\right)}{ 2-2\cos\left(k_{0}\right)-k_{0}\sin\left(k_{0}\right) }
  =
  \frac{\alpha}{c^{2}}\cdot 
  \frac{k_{0}}{ 2-k_{0}}.
\end{equation*}

The argument in~\cite{Kreiner2011a} and this paper uses an idea developed by Schwetlick and Zimmer for a
Fermi-Pasta-Ulam chain with nonconvex interaction potential, and no on-site potential~\cite{Schwetlick2009a}. This idea
is to represent the solution $u$ as $u = \up-r$ with explicitly given $\up$; then the analysis is reduced to a careful
investigation of the Fourier representation of $r$. Here, we will argue similarly and consider a ``profile'' function
$\up\in H^{2}_\loc(\RR)$. By profile function we mean that the function
$c^2\up''-\Delta_D \up+\alpha \up-\alpha \sgn(\up)$ satisfies
\begin{align}
  (1+x^2)(c^2\up''-\Delta_D \up+\alpha \up-\alpha \sgn(\up)) &\in L^2(\RR) \label{eq:profile1}\\
  \int_{\RR}\left[c^2\up''-\Delta_D \up+\alpha \up-\alpha \text{sgn}(\up)\right]\sin(k_0\cdot)dx &=0. \label{eq:profile2}
\end{align}
The former condition implies $c^2\up''-\Delta_D \up+\alpha \up-\alpha \text{sgn}(\up)\in L^1(\RR)$, so the latter
condition is well posed.  In addition, the function should be odd, $\sgn(\up(x))=\sgn(x)$ on $\RR$, vanishes at $x=0$,
satisfy $\up'(0)>0$ and $\liminf_{\abs{x}\rightarrow \infty}\abs{\up(x)}>0$, so that equation~\eqref{eq: cdn on rho}
below holds.

It is somewhat tedious but not difficult to find such a $\up$.  Specifically, we could use the profile function $\upa$
given above. However, we will use the solution to~\eqref{eq:introeq} with the special force $\psi'(x)=\sgn(x)$ as
profile. We therefore recall the existence result for this function.
\begin{theorem}[\protect{\cite[Theorem 4.1]{Kreiner2011a}}]
  \label{theo:KZ}
  Let $\psi'(x) = \sgn(x)$. Let $c$ be such that $c^{2}\in[0.83,1]$. Let $k_0$ be given by~\eqref{eq:k0} and $\alpha$
  be given by~\eqref{eq:alpha}. Then~\eqref{eq:introeq} has a solution $u=\upa-r$ with $u_{pa}$ given
  by~\eqref{eq:KZ-profile} with
  \begin{align*} 
    \sqrt{\frac{\pi}{2}}\: \abs{r(z)}
    & \leq
      \begin{cases}
        0.257 & \text{ for } c^{2}\in [0.9,1],\\
        0.339 & \text{ for } c^{2}\in [0.83,0.9],
      \end{cases} \\
    \intertext{ and }
    \sqrt{\frac{\pi}{2}}\: \abs{r'(z)}
    & \leq
      \begin{cases}
        0.43 & \text{ for } c^{2}\in [0.9,1],\\
        0.34 & \text{ for } c^{2}\in [0.83,0.9].
      \end{cases}
  \end{align*}
\end{theorem}
So below $\up$ will be the function $u$ of Theorem~\ref{theo:KZ}.  We are left with having to find
$r\in H^2_{\odd,\loc}(\RR)$ (that is, $r\in H^2_{\loc}(\RR)$ and $r(-x) = - r(x)$) such that $\up-r$ is a solution:
\begin{equation*}
  c^2(\up-r)''-\Delta_D (\up-r)+\alpha(\up-r) -\alpha\psi'(\up-r) =0,
\end{equation*}
and hence for $r$ 
\begin{equation*}
  c^2r''-\Delta_D r+\alpha r= c^2 \up''-\Delta_D  \up+\alpha \up-\alpha\psi'(\up-r),
\end{equation*}
which is an equation of the form 
\begin{equation*}
  c^2r''-\Delta_D r+\alpha r=Q, 
\end{equation*}
or $Lr = Q$ with nonlinear $Q$. We will employ Schauder's fixed point theorem to establish a solution to this
equation. The main result can be stated as follows.
\begin{theorem}
  \label{theo:main}
  For $\epsilon>0$, let the even function $\psi=\psi_\epsilon\in C^2(\RR)$ be such that $\psi_\epsilon'(x)=\sgn(x)$ for
  $\abs{x}\geq \epsilon$ and $\abs{\psi_\epsilon''(x)}\leq 2\epsilon^{-1}$ for $\abs{x} <\epsilon$. Let $k_0$ be given
  by~\eqref{eq:k0}, $\alpha$ be given by~\eqref{eq:alpha}.  Then there exists a range of subsonic velocities $c$ close
  to $1$ such that for these velocities, there exists a heteroclinic solution to~\eqref{eq:introeq}.
\end{theorem}
We remark that one of the conditions imposed on closeness of $c$ to $1$ is $c^{2}\in[0.83,1]$ as only in this case we
can build on the existence result Theorem~\ref{theo:KZ}.

Theorem~\ref{theo:main} is proved in the next section. We state one auxiliary statement for the equation $Lr = Q$.

\begin{proposition}
  \label{prop: L inverse}
  If $Q\in L^2_{odd}(\RR)$ satisfies 
  \begin{equation*}
    (1+x^2)Q\in L^2(\RR)~\text{ and }~
    \int_\RR Q(x)\sin(k_0 x)dx=0,
  \end{equation*}
  then, for all $c$ near enough to $1$, there exists a unique function $r\in H^2_{odd}(\RR)$ such that
  $Lr=c^2r''-\Delta_D r+\alpha r=Q$.  Moreover
  \begin{equation*}
    \norm{r}_{H^2(\RR)}:=\norm{(1+k^2)\widehat r}_{L^2(\RR)}
    \leq \{C_1+((4+\alpha)C_1+1)/c^2\}\norm{(1+x^2)Q}_{L^2(\RR)}
  \end{equation*}
  for some constant $C_1>0$ (independent of $c$ near $1$).
\end{proposition}
An extension of this result to functions $Q$ which are not necessarily odd can be found Proposition~\ref{prop: L
  inverse bis} in the Appendix.

\begin{proof}
The assumptions imply that $\widehat Q\in H^2(\RR,\CC)$, $\widehat Q(\pm k_0)=0$ and that there exists a unique
$r\in H^2_\odd(\RR)$ such that $c^2r''-\Delta_D r+\alpha r=Q$, namely
\begin{equation*}
  \widehat r(k)=\frac{\widehat Q(k)}{D(k)}~,~\text{for }k\in \RR.
\end{equation*}
As $Q$ is odd and real-valued, $i\widehat Q$ is odd and real-valued.  Therefore so are $i\widehat r$ and $r$.  Moreover,
\begin{multline*}
  \norm{\widehat Q\, '}_{L^{\infty}(\RR)}
  \leq \frac{1}{\sqrt{2\pi}}\int_\RR \frac{\abs{x}}{1+x^2}(1+x^2)\abs{Q(x)}dx
  \\\leq \frac{1}{\sqrt{2\pi}}\Big(\int_\RR \frac{x^2}{(1+x^2)^2}dx\Big)^{1/2}
  \norm{(1+x^2)Q}_{L^2(\RR)}
  = \frac{1}{2}\norm{(1+x^2)Q}_{L^2(\RR)}~,
\end{multline*}
(note that $(1/2)\arctan x-(1/2)x/(1+x^2)$ is a primitive of $x^2(1+x^2)^{-2}$).

Consider for a while $c=1$. For $\abs{k}\in[k_0/2,3k_0/2]\backslash\{k_0\}$, one gets by Cauchy's mean value theorem
applied to the real-valued functions $i \widehat Q$ and $D$
\begin{equation*}
  \abs{\frac{\widehat Q(k)}{D(k)}}
  \leq \sup_{\abs{s}\in[k_0/2,3k_0/2]\backslash\{k_0\}}
  \abs{\frac{\widehat Q\, '(s)}{D'(s)}}
  \leq \abs{D'(k_0/2)}^{-1}\frac{1}{2}\norm{(1+x^2)Q}_{L^2(\RR)}~.
\end{equation*}
For $\abs{k}\not \in [k_0/2,3k_0/2]$, one gets $\abs{D(k)}\geq \min\{\abs{D(k_0/2)},\abs{D(3k_0/2)}\}$.  Hence
\begin{multline*}
  \int_\RR\abs{\frac{\widehat Q(k)}{D(k)}}^2dk
  \leq 
  \max\{\abs{D(k_0/2)}^{-2},\abs{D(3k_0/2)}^{-2}\}
  \int_{\abs{k}\not \in [k_0/2,3k_0/2]}\abs{\widehat Q(k)}^2dk
  \\+2k_0\abs{D'(k_0/2)}^{-2}\frac{1}{4}\norm{(1+x^2)Q}_{L^2(\RR)}^2
  \\\leq\left(
    \max\{\abs{D(k_0/2)}^{-2},\abs{D(3k_0/2)}^{-2}\}+\frac 1 2 k_0\abs{D'(k_0/2)}^{-2}
  \right)
  \norm{(1+x^2)Q}_{L^2(\RR)}^2
  \\=C_1^2\norm{(1+x^2)Q}_{L^2(\RR)}^2~.
\end{multline*}
This estimate remains valid for all $c$ close to $1$ if we first increase slightly $C_1$.  As a consequence
\begin{multline*}
  c^2\norm{r''}_{L^2(\RR)}
  \leq (4+\alpha)\norm{r}_{L^2(\RR)}+\norm{Q}_{L^2(\RR)}
  \leq ((4+\alpha)C_1+1)\norm{(1+x^2)Q}_{L^2(\RR)}
\end{multline*}
and
\begin{multline*}
  \norm{r}_{H^2(\RR)}=\norm{(1+k^2)\widehat r}_{L^2(\RR)}
  \leq \norm{r}_{L^2(\RR)} +\norm{r''}_{L^2(\RR)}
  \\ \leq \{C_1+((4+\alpha)C_1+1)/c^2\}\norm{(1+x^2)Q}_{L^2(\RR)}~.
\end{multline*}
\end{proof}

\section{Proof of Theorem~{\protect\ref{theo:main}}}
\label{sec:Proof-Theor}

\subsection{Preliminaries}
\label{sec:Preliminaries}

We now turn to the proof of Theorem~\ref{theo:main}. We seek a solution to~\eqref{eq:introeq},
\begin{equation*}
  c^2u''-\Delta_D u+\alpha u-\alpha \psi'(u)=0. 
\end{equation*}
By assumption, $\psi\in C^2(\RR)$ is even and for its derivative it holds that $\psi'=\sgn$ outside a bounded set. We
split the solution $u$ sought to~\eqref{eq:introeq} as
\begin{equation}
  \label{eq: r}
  %c^2u''-\Delta_D u+\alpha u-\alpha \psi'(u)=0 \text{ with }
  u=\up+\beta \uodd+\gamma \sin(k_0\cdot)-r, 
\end{equation}
where the profile function $\up\in H^{2}_\loc(\RR)$ is odd and satisfies properties~\eqref{eq:profile1}
and~\eqref{eq:profile2}. Further, $\gamma\in \RR$ is assumed to be sufficiently close to $0$, and
$\uodd\in H^{2}_\loc(\RR)$ is an odd function such that for each $l=0,1,2$,
\begin{equation}
  \label{eq:u odd}
  (1+x^2)\frac{\mathrm{d}^l}{\mathrm{d}x^l}
  (\uodd(x)-\text{sgn}(x)\cos(k_0x))
  \in L^2(\RR\backslash[-1,1]). 
\end{equation}
For example, one can choose $\uodd$ to agree with $\sgn(x)\cos(k_0x)$ outside a bounded interval.  It is not hard to
give an explicit representation for $\uodd$, whereas $\up$ is the solution given by Theorem~\ref{theo:KZ}; the task is
then to find the corrector $r\in H^2_\odd(\RR)$ such that $u$ as in~\eqref{eq: r} solves~\eqref{eq:introeq}. The
periodic term $\gamma \sin(k_0\cdot)$ is separated from $\up$ for mere convenience; obviously this term could be added
to $\up$ and then $\tup:=\up+\gamma\sin(k_0\cdot)$ satisfies~\eqref{eq:profile1} and~\eqref{eq:profile2} and could
replace $\up$.

With this notation, we can now restate Theorem~\ref{theo:main} in a more detailed form we are going to establish.
\begin{theorem}
  \label{theo:main-2}
  For $\epsilon>0$, let the even function $\psi=\psi_\epsilon\in C^2(\RR)$ be such that $\psi_\epsilon'(x)=\sgn(x)$ for
  $\abs{x}\geq \epsilon$ and $\abs{\psi_\epsilon''(x)}\leq 2\epsilon^{-1}$ for $\abs{x} <\epsilon$. Let $k_0$ be given
  by~\eqref{eq:k0}, $\alpha$ be given by~\eqref{eq:alpha}.  Then there exists a range of subsonic velocities $c$ with
  $c^2 \geq 0.83$ such that a heteroclinic solution to~\eqref{eq:introeq} exists, in the following sense. Let the odd
  function $\up\in H^2_\loc(\RR)$ be the solution to the equation $c^2u''-\Delta_D u +\alpha u-\alpha\text{sgn}(u)=0$
  of Theorem~\ref{theo:KZ}, and let the odd function $\uodd\in H^2_\loc(\RR)$ satisfy \eqref{eq:u odd}.

  Then for all $\abs{\gamma}$ and $\rho>0$ small enough, there exists $\epsilon_0>0$ satisfying the following
  property. For every $\epsilon\in(0,\epsilon_0)$, there exists $r\in H^2_\odd(\RR)$ and $\beta\in \RR$ such that
  $\norm{r}_{H^2(\RR)}<\rho$ and $u := \up+\beta \uodd+\gamma\sin(k_0\cdot)-r$ is a solution to~\eqref{eq:introeq},
  \begin{multline}
    \label{eq:eqn-split}
    c^2(\up+\beta \uodd+\gamma\sin(k_0\cdot)-r)''
    -\Delta_D (\up+\beta \uodd+\gamma\sin (k_0\cdot)-r)
    \\
    +\alpha(\up+\beta \uodd+\gamma\sin(k_0\cdot)-r)
    -\alpha\psi_\epsilon'(\up+\beta \uodd+\gamma \sin(k_0\cdot)-r) 
    =0.
  \end{multline}
\end{theorem}
Theorem~\ref{theo:main} follows immediately once Theorem~\ref{theo:main-2} is established, and the rest of the article
is devoted to the proof of Theorem~\ref{theo:main-2}.

We start the proof by considering the linear operator $L$ of~\eqref{eq: L} with $\alpha$ as in~\eqref{eq:alpha} and $c$
being slightly subsonic. Specifically, we first study the equation $Lr=Q$ under the hypothesis
$\int_{\RR}Q(x)\sin(k_0 x) dx =0$, with $k_0=\pi/2$. Roughly speaking, in the equation $Lr=Q$, the right-hand side is
replaced by a new expression $Q$ depending on $\uodd$ and a real parameter $\beta$ chosen so that
$\int_{\RR}Q(x)\sin(k_0 x) dx =0$.

\begin{lemma}
  \label{lem:same-sign}
  Let $\up$ be the solution to the special case $\psi'(x) = \sgn(x)$ recalled in Theorem~\ref{theo:KZ}. There exists
  $\rho>0$ such that, for all $r$ in the ball $\overline{B(0,\rho)}\subset H^2_\odd(\RR)$, $\sgn(\up(x)-r(x))=\sgn(x)$
  on $\RR$.
\end{lemma}

\proof Recall the Sobolev estimates
\begin{multline*}
  \norm{r}_{L^\infty(\RR)}
  \leq \frac 1 {\sqrt{2\pi}}\int_{\RR}\frac{1}{1+k^2}(1+k^2)\abs{\widehat r}dk
  \\
  \leq \frac 1 {\sqrt{2\pi}}\sqrt{\int_{\RR}\frac{1}{(1+k^2)^2}dk}\norm{r}_{H^2(\RR)}
  = \frac 1 2 \norm{r}_{H^2(\RR)}
\end{multline*}
and
\begin{multline*}
  \norm{r'}_{L^\infty(\RR)}
  \leq \frac 1 {\sqrt{2\pi}}\int_{\RR}\frac{\abs{k}}{1+k^2}(1+k^2)\abs{\widehat r}dk
  \\
  \leq \frac 1 {\sqrt{2\pi}}\sqrt{\int_{\RR}\frac{k^2}{(1+k^2)^2}dk}\norm{r}_{H^2(\RR)}
  = \frac 1 2 \norm{r}_{H^2(\RR)}~.
\end{multline*}
By symmetry, it suffices to consider positive $x$. Hence it suffices to choose $\rho_0>0$ such that there is a point
$x_0 \in(0,1]$ such that
\begin{equation}
  \label{eq: cdn on rho}
  \up(x) >\rho_0/2 \text{ for  } x > x_0 \text{ and } \up'(x)> \rho_0/2 \text{ for every } x \in [0,x_0).
\end{equation}
Since $\up$ satisfies this property for some $\rho_0$, so the claim follows for any $\rho\in(0,\rho_0)$. \qed

Throughout this article, we will assume $\rho\in(0,\rho_0)$. We also assume that $\epsilon<\rho_0/6$, so that
$\psi'(s)=\sgn(s)$ for all $\abs{s}\geq \rho_0/6$.

If we add the requirement on $\beta$, $\gamma$ and $r$ that the condition
\begin{equation*}
  \abs{\beta \uodd(x)+\gamma\sin(k_0x)-r(x)}\leq \frac 2 3 \abs{\up(x)}
\end{equation*}
is fulfilled for all $x\in \RR$, the solving~\eqref{eq:introeq} with the \emph{ansatz}~\eqref{eq: r} is equivalent to
solving
\begin{equation}
  \label{eq: r-reformulated}
  c^2u''-\Delta_D u+\alpha u-\alpha \partial_1\Psi(u,x)=0 \text{ with }
  ~u=\up+\beta \uodd+\gamma \sin(k_0\cdot)-r,
\end{equation}
where $\Psi\colon\RR^2\rightarrow \RR$ satisfies 
\begin{equation*}
  \begin{cases}
    \Psi(u,x)=\psi(u) &  \text{ for } \abs{x}\leq 1, \\
    \Psi(u,x)=\sgn(x)u & \text{ for } \abs{x}\geq 1.
  \end{cases}
\end{equation*}
We prove the existence of a solution using Schauder's fixed point theorem.

\subsection{Application of Schauder's fixed point theorem}
\label{sec:Appl-Scha-fixed}

In this section, we prove the existence of a solution of a slightly relaxed problem, Equation~\eqref{eq r rewritten}
below, under fairly abstract assumptions, notably~\eqref{eq:C1},~\eqref{eq:C2} in Theorem~\ref{thm: Schauder} below.
The following sections then establish that the original problem can be cast in the setting studied here. 

Specifically, consider a modification~\eqref{eq:eqn-split} for $r\in H^2_\odd(\RR)$ and $\beta\in \RR$, and recall
$\psi'(u(x)) = \partial_1\Psi(u(x),x)$ for the function $u$ we have in mind,
\begin{multline}
  \label{eq:eqn-split-psi-xi}
  c^2(\up+\beta \uodd+\gamma\sin(k_0\cdot)-r)''
  -\Delta_D (\up+\beta \uodd+\gamma\sin (k_0\cdot)-r)
  \\
  +\alpha(\up+\beta \uodd+\gamma\sin(k_0\cdot)-r)
  -\alpha\partial_1\Psi(\up+\xi(\beta) \uodd+\gamma \sin(k_0\cdot)-r,x) 
  =0; 
\end{multline}
here the new ingredient is a function $\xi\in C^1(\RR)$ with $\norm{\xi'}_{L^\infty(\RR)}<\infty$. Thus, in a first
step, we replaced $\beta$ by $\xi(\beta)$ in the nonlinear term. As $\xi'$ is assumed to be bounded, the function $\xi$
allows us to control the nonlinear term without restrictions on the size of $\beta$.  In a second step, we shall assume
that $\xi$ is the identity near $0$ and show that the relevant values of $\beta$ are sufficiently close to $0$, so that
$\xi(\beta)=\beta$ for these values of $\beta$.

The assumptions in this Subsection are as follows. We recall $k_0$ is given by~\eqref{eq:k0}, $\alpha$ is given
by~\eqref{eq:alpha}, and $c$ is close to $1$.  We have seen that then the dispersion function in ~\eqref{eq:disp} has
exactly two simple roots $\pm k_0$. Furthermore, for the linear operator given in~\eqref{eq: L}, $L\sin(k_0\cdot)=0$.
Let $\Psi\colon\RR^2\rightarrow \RR$ be of class $C^2$ with respect to the first variable, $\Psi$, $\partial _1 \Psi$
and $\partial^2_{11}\Psi$ be measurable with respect to the second variable, $\partial_1\Psi$ be odd and
\begin{equation}
  \label{eq: bound on psi}
  \abs{\partial_{11}^2 \Psi(s,x)}\leq \frac{\mu}{(1+ x^2)^{3/2}}
\end{equation}
for some constant $\mu>0$.  Note that
\begin{equation*}
  (1+x^2)\frac{1}{(1+x^2)^{3/2}}\in L^2(\RR).
\end{equation*}
The size of $\mu$ does not matter in what follows (in particular, it is not assumed to be small).

We recall that the parameter $\gamma$ is real-valued, and that $\up$ is a given odd function in $H^{2}_\loc(\RR)$
satisfying
\begin{multline}
  \label{eq: unif bound}
  \sup_{\beta\in \RR}\left\lVert(1+x^2)^{3/2}
    \Big
    (c^2\up''-\Delta_D \up+\alpha \up
  \right.\\\left.
    -
    \alpha \partial_1\Psi\big(\up+\xi(\beta) \uodd+\gamma \sin(k_0x),x\big)
    \Big)\right\rVert_{L^{\infty}(\RR)}
  <\infty.
\end{multline}
The odd function $\uodd\in H^{2}_\loc(\RR)$ satisfies~\eqref{eq:u odd}. Thus, since $L\cos(k_0\cdot)=0$,
\begin{equation*}
  (1+x^2)L\uodd=(1+x^2)(c^2\uodd''-\Delta_D\uodd+\alpha \uodd)\in L^2_\odd(\RR).
\end{equation*}
It follows that the map
\begin{multline*} 
  (r,\beta)\rightarrow \Gamma(r,\beta)=
  (1+x^2)\Big(c^2\up''-\Delta_D \up+\alpha \up
  \\
  -\alpha \partial_1\Psi\big(\up+\xi(\beta) \uodd+\gamma \sin(k_0x)-r,x\big)
  \Big)\in L^{2}_\odd(\RR)
\end{multline*}
is well-defined on $H^2_\odd(\RR)\times \RR$ and of class $C^1$.

\begin{lemma}
  \label{lem:Gamma cpt}
  The map $\Gamma\colon H^2_\odd(\RR)\times \RR \to L^{2}_\odd(\RR)$ is compact.
\end{lemma}

\begin{proof}
The map can be written as
\begin{multline*} 
  \Gamma(r,\beta)=
  (1+x^2)\Big(c^2\up''-\Delta_D \up+\alpha \up
  -\alpha \partial_1\Psi\big(\up+\xi(\beta) \uodd+\gamma \sin(k_0x),x\big)\Big)
  \\ +\alpha (1+x^2)\int_0^1\partial^2_{11}
  \Psi\big(\up+\xi(\beta) \uodd+\gamma \sin(k_0x)-s r,x\big)r\, ds,
\end{multline*}
which is the sum of two terms in $L^2(\RR)$ (see~\eqref{eq: bound on psi} and~\eqref{eq: unif bound}).  Let
$\{(r_n,\beta_n)\}\subset H^2_\odd(\RR)\times \RR$ be a bounded sequence. We verify that $\{\Gamma(r_n,\beta_n)\}$ has
a Cauchy subsequence in $L^2_\odd(\RR)$. Let $\varepsilon>0$.

Since $\xi$ is continuous on $\RR$, the sequence $\{\xi_n\}:=\{\xi(\beta_n)\}$ is bounded.  Taking a convergent
subsequence $\{\xi_{n_k}\}$, equation~\eqref{eq: unif bound} and the dominated convergence theorem ensure that the
first term of $\Gamma(r_{n_k},\beta_{n_k})$ converges as $k\rightarrow \infty$. Hence, for $k,l$ large enough,
\begin{multline*}
  \Big\lVert
  (1+x^2)\Big(c^2\up''-\Delta_D \up+\alpha \up-\alpha \partial_1\Psi\big(\up+\xi(\beta_{n_k}) \uodd+\gamma \sin(k_0x),x\big)\Big)
  \\-
  (1+x^2)\Big(c^2\up''-\Delta_D \up+\alpha \up-\alpha \partial_1\Psi\big(\up+\xi(\beta_{n_l}) \uodd+\gamma \sin(k_0x),x\big)\Big)
  \Big\Vert_{L^2(\RR)}<\frac\varepsilon2.
\end{multline*}

To deal with the second term, we split $\RR$ in two parts, namely $I_\varepsilon := [-x_{\varepsilon},x_{\varepsilon}]$
and its complement in $\RR$, where $x_\varepsilon>0$ is large.  The motivation for this split is that many Sobolev
embeddings are compact on an bounded interval, whereas the second term can be assumed as small as needed when
restricted to the complement of $I_\varepsilon$.  More precisely, given $\varepsilon>0$, choose $x_\varepsilon$ large
enough so that for all $k$
\begin{equation*}
  \alpha
  \left\Vert
    (1+x^2)\int_0^1\partial^2_{11}\Psi\big
    (\up+\xi_{n_k} \uodd+\gamma \sin(k_0x)
    -s r_{n_k},x\big)r_{n_k}\,ds\right\rVert_{L^2(\RR\backslash I_\varepsilon)} %[-x_\epsilon,x_\epsilon])}
  <\frac\varepsilon8
\end{equation*}
(see~\eqref{eq: bound on psi}).  Using the compact embedding
$H^2(-x_\varepsilon,x_\varepsilon) \subset C[-x_\varepsilon,x_\varepsilon]$, by taking a further subsequence if
necessary, we can assume that $\{r_{n_k}\}$ converges in $ C[-x_\varepsilon,x_\varepsilon]$.  It follows, again from
the dominated convergence theorem, that
\begin{equation*}\alpha (1+x^2)
\int_0^1\partial^2_{11}\Psi\big(\up+\xi_{n_k} \uodd+\gamma \sin(k_0x)
-s r_{n_k},x\big)r_{n_k}\,ds\end{equation*} 
converges in $L^2(-x_\varepsilon,x_\varepsilon)$.
Hence, for $k,l$ large enough,
\begin{multline*}
  \Big\lVert\alpha (1+x^2)\int_0^1\partial^2_{11}
  \Psi\big(\up+\xi_{n_k} \uodd+\gamma \sin(k_0x)
  -sr_{n_k},x\big)r_{n_k}\,ds
  \\-\alpha (1+x^2)\int_0^1\partial^2_{11}\Psi
  \big(\up+\xi_{n_l} \uodd+\gamma \sin(k_0x)
  -s r_{n_l},x\big)r_{n_l}\,ds\Big\rVert_{L^2(\RR)}<\epsilon/2.
\end{multline*}
Thus  $\{\Gamma(r_{n_k},\beta_{n_k})\}$ is a Cauchy subsequence.
\end{proof}

By~\eqref{eq:uo ortho} of Proposition~\ref{eq: orthogonality} in the Appendix,
\begin{equation*}
  \int_{\RR}\left(c^2\uodd''-\Delta_D \uodd+\alpha \uodd\right) \sin(k_0\cdot ) dx=-2c^2k_0+2<0
\end{equation*}
if $c>k_0^{-1/2}$.  Assume that, for all $r$ in some subset of $H^2_\odd(\RR)$ and all $\beta \in \RR$,
\begin{multline*}
  \abs{\int_{\RR}\alpha\partial_{11}^2 
    \Psi\big(\up+\xi(\beta) \uodd+\gamma\sin(k_0\cdot)-r,\cdot\big)
    \xi'(\beta)\uodd\sin(k_0\cdot)dx
  }
  \\
  \leq C
  \abs{\int_{\RR}\left(c^2\uodd''-\Delta_D\uodd+\alpha \uodd\right)
    \sin(k_0\cdot) dx}=C \,2(c^2k_0-1)
\end{multline*}
for some constant $C \in [0,1)$. Then for fixed $r$ in the given subset, the equation
\begin{multline*}
  \int_\RR\Big(c^2\big(\up+\beta \uodd+\gamma\sin(k_0\cdot)\big)''
  -\Delta_D \big(\up+\beta \uodd+\gamma\sin (k_0\cdot)\big)
  +\alpha\big(\up+\beta \uodd+\gamma\sin(k_0\cdot)\big)
  \\-\alpha\partial_1\Psi\big(\up+\xi(\beta) \uodd+\gamma \sin(k_0\cdot)-r,x
  \big) \Big)
  \sin(k_0\cdot)dx
  =0
\end{multline*}
can uniquely be solved for $\beta$ as a $C^1$-function of $r$, $\beta=\beta(r)$, thanks to Banach's fixed point theorem
and the implicit function theorem.

\begin{lemma}
\label{lem:beta bounded}
  The map $r\rightarrow \beta(r)$ is bounded on bounded sets.
\end{lemma}

\begin{proof}
The proof of Lemma~\ref{lem:Gamma cpt} shows an additional property, namely that the map
$(r,\beta)\rightarrow \Gamma(r,\beta)$ is bounded on every set on which the $r$-component is bounded.  As a
consequence, by definition of $\beta=\beta(r)$,
\begin{multline*}
  2(c^2 k_0-1) \beta=
  -\beta\int_\RR (c^2\uodd-\Delta_D\uodd+\alpha \uodd)\sin(k_0x)dx=
  \\\int_\RR\left[c^2\up''-\Delta_D \up+\alpha \up
  -\alpha\partial_1\Psi\big(\up+\xi(\beta) \uodd+\gamma \sin(k_0\cdot)-r,\cdot
  \big)\right]\sin(k_0 x)dx
\end{multline*}
and 
\begin{equation}
  \label{eq: beta(r)}
  \beta= \frac{
    \int_\RR\left[c^2\up''-\Delta_D \up+\alpha \up
      -\alpha\partial_1\Psi\big(\up+\xi(\beta) \uodd+\gamma \sin(k_0\cdot)-r,\cdot
      \big)\right]\sin(k_0 x)dx
  }{2(c^2k_0-1)}.
\end{equation}
The map $r\rightarrow \beta(r) =\frac 1 2(c^2k_0-1)^{-1}\int_{\RR}\Gamma(r,\beta(r))(1+x^2)^{-1}\sin(k_0x)dx$ is thus
bounded on bounded sets.
\end{proof}

Hence the problem can be written as $c^2r''-\Delta_D r+\alpha r=Q$, with
\begin{multline*}
  Q=c^2\big(\up+\beta(r) \uodd+\gamma\sin(k_0\cdot)\big)''
  -\Delta_D \big(\up+\beta(r) \uodd+\gamma\sin (k_0\cdot)\big)
  \\+\alpha\big(\up+\beta(r) \uodd+\gamma \sin(k_0\cdot)\big) 
  -\alpha\partial_1\Psi\big(\up+\xi(\beta(r)) \uodd+\gamma \sin(k_0\cdot)-r,\cdot
  \big)
  \\=\beta(r)\big(c^2\uodd''-\Delta_D \uodd+\alpha \uodd\big)
  +(1+x^2)^{-1}\Gamma(r,\beta(r))
  \in L^2_\odd(\RR)
\end{multline*}
and $\int_{\RR} Q(x)\sin(k_0 x)dx=0$ by definition of $\beta(r)$.  On the other hand, if $Q\in L^2(\RR)$ is odd with
\begin{equation*}
  (1+x^2)Q\in L^2(\RR)~\text{ and }~
  \int_\RR Q(x)\sin(k_0 x)dx=0,
\end{equation*}
we saw in Proposition~\ref{prop: L inverse} that there exists a unique odd $r=L^{-1}Q\in H^2(\RR)$ such that
$Lr=c^2r''-\Delta_D r+\alpha r=Q$.  Moreover
\begin{equation*}
  \norm{L^{-1}Q}_{H^2(\RR)}
  =\norm{r}_{H^2(\RR)}
  \leq \{C_1+((4+\alpha)C_1+1)/c^2\}\norm{(1+x^2)Q}_{L^2(\RR)}
\end{equation*}
for some constant $C_1>0$.

The problem~\eqref{eq:eqn-split-psi-xi} studied in this Subsection can be written as
\begin{multline}
  r=L^{-1}Q=
  L^{-1}\Big(c^2(\up+\beta \uodd+\gamma\sin(k_0\cdot))''
  -\Delta_D (\up+\beta \uodd+\gamma\sin (k_0\cdot))
  \\+\alpha(\up+\beta \uodd+\gamma \sin(k_0\cdot)) 
  -\alpha\partial_1\Psi\big(\up+\xi(\beta) \uodd+\gamma \sin(k_0\cdot)-r,x\big)\Big)
  \label{eq r rewritten}
\end{multline}
with $\beta=\beta(r)$.

\begin{theorem}
  \label{thm: Schauder}
  Let $\xi$ be in $C^1(\RR)$ with $\norm{\xi'}_{L^\infty(\RR)}<\infty$. Let $k_0$ be as in~\eqref{eq:k0} and $\alpha$
  given by~\eqref{eq:alpha}, Let $\Psi\colon\RR^2\rightarrow \RR$ be of class $C^2$ with respect to the first variable,
  let $\Psi$, $\partial _1 \Psi$ and $\partial^2_{11}\Psi$ be measurable with respect to the second variable, and
  $\partial_1\Psi$ be odd. Assume that the hypotheses~\eqref{eq:u odd}, ~\eqref{eq: bound on psi} and \eqref{eq: unif
    bound} hold.  Suppose that there exists an open ball $B(0,\rho)\subset H^2_\odd(\RR)$ such that
  \begin{multline}
    \sup_{r\in \overline{B(0,\rho)},\, \beta\in \RR}
    \abs{\int_{\RR}\alpha\partial_{11}^2 
      \Psi\big(\up+\xi(\beta) \uodd+\gamma\sin(k_0\cdot)-r,\cdot\big)
      \xi'(\beta)\uodd\sin(k_0\cdot)dx
    }
    \\<2(c^2k_0-1)
    \tag{C1}
    \label{eq:C1}
  \end{multline}
  and 
  \begin{equation}
    \label{eq:C2}
    \sup_{r\in \overline{B(0,\rho)}}\norm{(1+x^2)Q(r)}_{L^2(\RR)}
    < \{C_1+((4+\alpha)C_1+1)/c^2\}^{-1}\,\rho
    .
    \tag{C2}
  \end{equation}
  Then there exists a solution $r\in B(0,\rho)$ to~\eqref{eq r rewritten}.
\end{theorem}

\begin{proof}
For all $r\in \overline{B(0,\rho)}$, $Q=Q(r)$ is well defined  with values in

\begin{equation*}
  Z=\left\{f\in L^2_\odd(\RR):(1+x^2)f\in L^2(\RR),~
    \int_{\RR}f(x)\sin(k_0x)dx=0\right\}
\end{equation*}
and completely continuous in $r$ (that is, continuous and compact). The map $r\rightarrow L^{-1}Q(r)$ sends
$\overline{B(0,\rho)}$ into $B(0,\rho)$ and is completely continuous.  The Schauder fixed point theorem gives a
solution $r\in \overline{B(0,\rho)}$ to the equation $r=L^{-1}Q(r)$, and in fact $r\in B(0,\rho)$.
\end{proof}

\subsection{On the verification of condition\protect{~\eqref{eq:C2}}}
\label{sec:verif-cond}

In this section, we establish one condition,~\eqref{eq:C2 prime} below, for the verification of condition~\eqref{eq:C2}
in Theorem~\ref{thm: Schauder}. This simpler condition will then be shown in Subsection~\ref{sec:Verif-cond-Theor} to
hold under the assumptions of Theorem~\ref{theo:main-2}.

By the formula~\eqref{eq: beta(r)} for $\beta=\beta(r)$ and
\begin{multline*}
  Q=(c^2\uodd''-\Delta_D \uodd+\alpha \uodd)\beta
  +c^2\up''-\Delta_D \up+\alpha \up
  \\-\alpha\partial_1\Psi(\up+\xi(\beta) \uodd+\gamma \sin(k_0\cdot)-r,\cdot),
\end{multline*}
one has
\begin{multline*}
  \norm{(1+x^2)Q}_{L^2(\RR)}
  \leq \norm{(1+x^2)\left(c^2\uodd''-\Delta_D \uodd+\alpha \uodd\right)}_{L^2(\RR)}
  \\
  \times
  \frac{
    \abs{\int_\RR\Big\{c^2\up''-\Delta_D \up+\alpha \up
      -\alpha\partial_1\Psi(\up+\xi(\beta) \uodd+\gamma \sin(k_0\cdot)-r,\cdot)\Big\}
      \sin(k_0 x)dx}
  }{2(c^2k_0-1)}
  \\+\norm{(1+x^2)(c^2\up''-\Delta_D \up+\alpha \up
    -\alpha\partial_1\Psi(\up+\xi(\beta) \uodd+\gamma \sin(k_0\cdot)-r,\cdot))
  }_{L^2(\RR)}~.
\end{multline*}
Hence condition~\eqref{eq:C2} is ensured by the following condition
\begin{multline*}
  \sup_{r\in \overline{B(0,\rho)}}\left\{
  \norm{(1+x^2)\left(c^2\uodd''-\Delta_D \uodd+\alpha \uodd\right)}_{L^2(\RR)}
  \frac{
    \norm{(1+x^2)^{-1}\sin(k_0\cdot)}_{L^2(\RR)}
  }{2(c^2k_0-1)}+1\right\}
  \\\times
  \norm{(1+x^2)\Big(c^2\up''-\Delta_D \up+\alpha \up
    -\alpha\partial_1\Psi
    \big(\up+\xi(\beta) \uodd+\gamma \sin(k_0\cdot)-r,\cdot\big)\Big)}_{L^2(\RR)}
  \\ <\frac 1  {C_1+((4+\alpha)C_1+1)/c^2}\, \rho,
\end{multline*}
which in turn is ensured by the condition
\begin{multline*}
  \sup_{r\in \overline{B(0,\rho)}}
  \left\lVert(1+x^2)\Big(c^2\up''-\Delta_D \up+\alpha \up
  \right.\\\left.
    -\alpha\partial_1\Psi\big(\up+\xi(\beta) \uodd+\gamma \sin(k_0\cdot)-r,\cdot
    \big)\Big)\right\rVert_{L^2(\RR)}
  \\ <\frac  { \{C_1+((4+\alpha)C_1+1)/c^2\}^{-1} \rho}
  {\norm{(1+x^2)(c^2\uodd''-\Delta_D \uodd+\alpha \uodd)}_{L^2(\RR)}
    \sqrt {\pi/8 }(c^2k_0-1)^{-1}+1}~.
\end{multline*}

If $\up$ is a particular solution to the ``unperturbed'' equation
$c^2\up''-\Delta_D \up+\alpha \up-\alpha S(\up,\cdot)=0$ for some function $S$, if
\begin{multline}
  \norm{(1+x^2)(\alpha S(\up,\cdot)-\alpha\partial_1\Psi(\up,\cdot))}_{L^2(\RR)}
  \\+\sup_{r\in \overline{B(0,\rho)}}
  \left\lVert(1+x^2)\Big(\alpha\partial_1\Psi(\up,\cdot)
    ~~~~~~~~~~~~~~~~~~\right.\\\left.~~~~~~~~~~~~~~~
    -\alpha\partial_1\Psi\big(\up+\xi(\beta(r)) \uodd+\gamma \sin(k_0\cdot)-r,\cdot
    \big)\Big)\right\rVert_{L^2(\RR)}
  \\ <\frac  {\left(C_1+((4+\alpha)C_1+1)/c^2\right)^{-1} \rho}
  {\norm{(1+x^2)(c^2\uodd''-\Delta_D \uodd+\alpha \uodd)}_{L^2(\RR)}
    \sqrt {\pi/8 }(c^2k_0-1)^{-1}+1}
  \tag{C2'}
  \label{eq:C2 prime}
\end{multline}
and if the condition~\eqref{eq:C1} holds true, then the ``perturbed'' problem, in which $S$ is replaced by
$\partial_1\Psi$ and the parameter $\gamma$ can be chosen in $\RR$, has a solution $r\in B(0,\rho)$.

\subsection{Verification of the conditions in Theorem~\protect{\ref{thm: Schauder}}}
\label{sec:Verif-cond-Theor}

In this section, we prove Theorem~\ref{theo:main-2}. We have to show that the assumptions made there imply those of
Theorem~\ref{thm: Schauder}, and show that $\xi$ can be chosen to be the identity in the region of interest.

We make the same assumptions on $k_0$, $\alpha$, $\uodd$ and $\up$ as in Theorem~\ref{theo:main-2}. In particular, the
chosen $\up$ is such that $\up'(0)>0$,
\begin{equation*}
  \int_{\RR}\left(c^2\up''-\Delta_D \up+\alpha \up-\alpha \text{sgn}(\up)\right)\sin(k_0\cdot)dx=0,
\end{equation*}
and
\begin{equation*}
  \norm{(1+x^2)^{3/2}\left(c^2\up''-\Delta_D \up+\alpha \up-\alpha \text{sgn}(\up)\right)}_{L^{\infty}(\RR)}
  <\infty.
\end{equation*}
Let $\rho_0>0$ satisfy \eqref{eq: cdn on rho}; then $\abs{\up(x)}>\rho_0/2$ for all $\abs{x}\geq 1$.

\begin{lemma}
  In the setting of this subsection, $\xi$ can be chosen such that the solution given by Theorem~\ref{thm: Schauder}
  solves~\eqref{eq:eqn-split}.
\end{lemma}

\begin{proof}
In Equation~\eqref{eq:eqn-split-psi-xi}, we choose $\xi$ such that it is the identity function in a neighbourhood of 
$\beta=0$ and
\begin{equation*}
  % \label{eq: bound on xi infty}
  \norm{\xi}_{L^\infty(\RR)}\abs{\uodd(x)}\leq \frac 1 3 \abs{\up(x)} \text{ for all } x\in \RR~.
\end{equation*}
If $\abs{\gamma}$ and $\norm{r}_{H^2(\RR)}$ are small enough, then for every $x \in \RR$
\begin{equation}
  \label{eq: new reference}
  \abs{\up(x)+\xi(\beta) \uodd(x)+\gamma\sin(k_0x)-r(x)}\geq \frac 1 3 \abs{\up(x)}
\end{equation}
and thus
\begin{equation*}
  \partial_1\Psi\Big(\up+\xi(\beta) \uodd+\gamma \sin(k_0\cdot)-r\,,\,x\Big) 
  =\psi'\Big(\up+ \xi(\beta) \uodd+\gamma \sin(k_0\cdot)-r\Big).
\end{equation*}
Hence, we will obtain the solution $u=\up+\beta \uodd+\gamma\sin(k_0\cdot)-r$ to
\begin{equation*}
  c^2u''-\Delta_D u+\alpha u-\alpha \psi'(u)=0
\end{equation*}
if, in addition, $\xi(\beta)=\beta$.
\end{proof}

\begin{lemma}
  Under the assumptions of Theorem~\ref{theo:main-2}, assumption~\eqref{eq: bound on psi} of Theorem~\ref{thm:
    Schauder} holds.
\end{lemma}

\begin{proof}
  This is immediate; recall that $\psi\in C^2(\RR)$ is even, with $\psi'(s)=\text{sgn}(s)$ outside a bounded set.  By
  reducing $\epsilon$ if necessary we can assume that $\psi'(s)=\sgn(s)$ for all $\abs{s} \geq \rho_0/6$.  Then
  $\Psi\colon\RR^2\rightarrow \RR$ satisfies $\Psi(u,x)=\psi(u)$ for $\abs{x}\leq 1$ and $\Psi(u,x)=\text{sgn}(x)u$ for
  $\abs{x}\geq 1$.
\end{proof}
 
\begin{lemma}
  Under the assumptions of Theorem~\ref{theo:main-2}, the assumptions~\eqref{eq: unif bound},~\eqref{eq:C1}
  and~\eqref{eq:C2 prime} hold.
\end{lemma}

\begin{proof}
We first establish the claim for~\eqref{eq:C1}.  Let us recall that $\psi$ such that
$\abs{\psi''(s)}\leq 2\epsilon^{-1}$ for $\abs{s} <\epsilon$ and $\psi''(s)=0$ otherwise, where $\epsilon>0$.  If
$\epsilon$ is small enough and $\abs{x}=6\epsilon/\up'(0)$, then
\begin{equation*}
  \abs{\up(x)}=\up'(0)\abs{x}(1+o(x))
  \geq\frac 1 2  \up'(0)\abs{x}\geq 3\epsilon
\end{equation*}
and thus $\abs{\up(x)}\geq 3\epsilon$ for all $\abs{x}\geq 6\epsilon/\up'(0)$ if $\epsilon$ is small enough.  Hence
\begin{equation*}
  \psi''\left(\up(x)+\xi(\beta) \uodd(x)+\gamma\sin(k_0x)-r(x)\right)=0
\end{equation*}
for all $\abs{x}\geq 6\epsilon/\up'(0)$ if $\abs{\gamma}$, $\norm{r}_{H^2(\RR)}$ and $\epsilon$ are small enough
(see~\eqref{eq: new reference}). Therefore
\begin{multline*}
  \abs{\int_{\RR}\alpha
    \psi''(\up+\xi(\beta) \uodd+\gamma\sin(k_0\cdot)-r)
    \xi'(\beta)\uodd\sin(k_0\cdot)dx
  }
  \\\leq \int_{-6\epsilon/\up'(0)}^{6\epsilon/\up'(0)}\alpha
  2\epsilon^{-1}\abs{\xi'(\beta)\uodd\sin(k_0\cdot)}dx
  \\
  \leq \alpha 2\epsilon^{-1}\norm{\xi'(\beta)\uodd}_{L^{\infty}(\RR)}\int_{-6\epsilon/\up'(0)}^{6\epsilon/\up'(0)}
  \abs{k_0x}dx
  \\\leq\alpha 2\epsilon^{-1}\norm{\xi'(\beta)\uodd}_{L^{\infty}(\RR)}
  k_0(6\epsilon/\up'(0))^2\rightarrow 0
\end{multline*}
as $\epsilon\rightarrow 0$, uniformly in $\beta \in \RR$ and $r\in \overline{B(0,\rho)}$ if $\abs{\gamma}$ and $\rho>0$
are small enough.  Hence~\eqref{eq:C1} holds true.  Assumption \eqref{eq: unif bound} can be verified similarly.

We now show that~\eqref{eq:C2 prime} is satisfied. We choose for $\up$ the solution of the degenerate problem
$c^2u''-\Delta_D u +\alpha u-\alpha\text{sgn}(u)=0$, see Theorem~\ref{theo:KZ}, and choose $\epsilon>0$ small enough so
that
\begin{multline*}
  \norm{(1+x^2)\left(\alpha\text{sgn}(\up)-\alpha\partial_1\Psi(\up,\cdot)\right)}_{L^2(\RR)}
  \\ <\frac  {\left( \{C_1+((4+\alpha)C_1+1)/c^2\right)^{-1} \rho}
  {2\norm{(1+x^2)(c^2\uodd''-\Delta_D \uodd+\alpha \uodd)}_{L^2(\RR)}
    \sqrt {\pi/8 }(c^2k_0-1)^{-1}+1}.
\end{multline*}
Then observe that, for all $r\in \overline{B(0,\rho)}$,
\begin{multline*}
  \norm{(1+x^2)\left(\alpha\partial_1\Psi(\up,\cdot)
    -\alpha\partial_1\Psi(\up+\xi(\beta(r)) \uodd+\gamma \sin(k_0\cdot)-r,\cdot)\right)}_{L^2(\RR)}
  \\ \leq \left\lVert(1+x^2)\alpha
    \sup_{\lambda\in[0,1]}\abs{\partial^2_{11}\Psi
      (\up+\lambda\xi(\beta(r)) \uodd+\lambda\gamma \sin(k_0\cdot)-
      \lambda r,\cdot)}\,\right.
  \\ ~~~~~~~~\times \left. \phantom{\sup_{\lambda\in[0,1]}}
    \abs{\xi(\beta(r)) \uodd+\gamma \sin(k_0\cdot)- r}\, \right\rVert_{L^2(\RR)}~.
\end{multline*}
Arguing as above,
\begin{multline*}
  \norm{(1+x^2)(\alpha\partial_1\Psi(\up,\cdot)
    -\alpha\partial_1\Psi(\up+\xi(\beta(r)) \uodd+\gamma \sin(k_0\cdot)-r,\cdot))}_{L^2(\RR)}
  \\ \leq \alpha 2\epsilon^{-1} \norm{(1+x^2)
    (\xi(\beta(r)) \uodd+\gamma \sin(k_0\cdot)- r)}_{L^2([-6\epsilon/\up'(0),6\epsilon/\up'(0)])}
  \rightarrow 0
\end{multline*}
as $\epsilon\rightarrow 0$, uniformly in $r\in \overline{B(0,\rho)}$ if $\abs{\gamma}$ and $\rho>0$ are small enough.

By Theorem~\ref{thm: Schauder}, there exists $r\in H^2_\odd(\RR)$ such that $\norm{r}_{H^2(\RR)}<\rho$ and
\begin{multline*}
  c^2(\up+\beta(r) \uodd+\gamma\sin(k_0\cdot)-r)''
  -\Delta_D (\up+\beta(r) \uodd+\gamma\sin (k_0\cdot)-r)
  \\+\alpha(\up+\beta(r) \uodd+\gamma\sin(k_0\cdot)-r)
  -\alpha\psi'(\up+ \xi(\beta(r)) \uodd+\gamma \sin(k_0\cdot)-r) 
  =0~.
\end{multline*}
We also get that $\beta(r)$ belongs to the neighbourhood of $0$ on which $\xi$ is the identity if
$\abs{\gamma},\rho,\epsilon$ are small enough.  Indeed, by~\eqref{eq: beta(r)},
\begin{multline*}
  \abs{\beta(r)}
  \leq \frac{
    \abs{\int_\RR\Big\{
      \alpha\text{sgn}(\up)
      -\alpha\partial_1\Psi(\up+\xi(\beta) \uodd+\gamma \sin(k_0\cdot)-r,\cdot)\Big\}\sin(k_0 x)dx}
  }{2(c^2k_0-1)}
  \\
  \leq \frac{1}{2(c^2k_0-1)}
  \int_{-6\epsilon/\up'(0)}^{6\epsilon/\up'(0)}\alpha\left(
    1+\frac {2}{\epsilon}\abs{\up+\xi(\beta) \uodd+\gamma \sin(k_0\cdot)-r}\right)
  k_0\abs{x}dx
  \\
  =O(1)
  \int_{-6\epsilon/\up'(0)}^{6\epsilon/\up'(0)}\abs{x}dx
  =O(\epsilon^2).
\end{multline*}

\end{proof}

\section{Two-transition solutions}
\label{sec:Two-transition-solut}

In this section, we show the existence of travelling waves starting in one well of the on-site potential, making a
transition to another well before returning to the first well. The on-site potential will be taken to be piecewise
quadratic, $\psi'(x)=\sgn(x)$, as in~\cite{Kreiner2011a}.  Also, we consider the same velocity regime
$c^{2}\in[0.83,1]$ as in that paper.

Our aim is to prove the existence of solutions representing two transitions between the two wells.  We construct the
solution similarly as in~\eqref{eq: r} for the case of a single transition, where the odd profile function $\up$ will
be replaced by an even profile function $\vp$, and similarly the odd function $\uodd$ will be replaced by an even
function $\ue$. That is, we use a decomposition of the form
\begin{equation}
  \label{eq: r tilde}
  u(x)=\vp(x)+\be \ue(x)+\tilde \gamma \cos(k_0x)-\tilde r(x) .
\end{equation}
Here $\vp$ is the primary profile, $\be$ a small coefficient scaling the contribution from $\ue$, $\tilde \gamma$ a
coefficient to be chosen later, and $\tilde r$ a (small) remainder.

We first turn the attention to $\vp$.
\begin{lemma}
  \label{lem:vp}
  Let $x_0\in(\pi/k_0)\ZZ=2\ZZ$ be positive. Then there exist an even profile 
$\vp\in H^{2}_\loc(\RR)$
 such that $\vp$ vanishes exactly at
  the two points $\pm x_0$. Furthermore,
  \begin{equation}
    \label{eq:vp-norm}
   \norm{(1+x^2)\left(L\vp-\alpha{\rm{sgn}}(\vp)\right) }_{L^2(\RR)}\to0
  \end{equation}
  as $x_0\to\infty$.
\end{lemma}

\begin{proof}
The odd solution $x\rightarrow \upa(x)-r(x)$ in~\cite{Kreiner2011a} (see~\eqref{eq:KZ-profile} and~\eqref{AandB})
converges in $H^2(z-2,z+2)$ as $\abs{z}\rightarrow \infty$ to the function
\begin{equation*}
  \text{sgn}(x)\Big(A+B-B\cos(k_0x) \Big), \text{ where $A+B=1$ and $\K=\KK$},
\end{equation*}
where the expression for $\K$ makes use of~\eqref{eq:k0} and~\eqref{eq:alpha}.

It is straightforward to see that $-\upa+r$ is also a single-transition solution to the solution to the problem with
piecewise quadratic on-site potential studied. We now introduce a two-transition profile $\vp$ by combining these two
single-transition solutions. Namely, for positive $x_0\in2\ZZ$, we define $\vp$ as
\begin{multline*}
  \vp(x):= \left(\frac 1 2+\lambda(x)\right)
  \left(\upa(x-x_0)-r(x-x_0\right) \\
  -\left(\frac 1 2 -\lambda(x)\right)\left(\upa(x+x_0)-r(x+x_0)\right),
\end{multline*}
where the step function $\lambda\in C^\infty(\RR,\RR)$ is odd and non-decreasing with $\lambda(x):=-1/2$ for $x\le-1$
and $\lambda(x):=1/2$ for $x\ge1$.

Obviously $\vp$ is even, piecewise $C^2$, and satisfies $L\vp-\alpha\text{sgn}(\vp)=0$ on $\RR\backslash[-2,2]$. To
show~\eqref{eq:vp-norm}, we thus only have to show that $\norm{L\vp-\alpha\text{sgn}(\vp)}_{L^2(-2,2)}$ tends to $0$ as
$x_0\rightarrow \infty$ with $x_0\in 2\ZZ$. We first deal with $x_0\in 4\ZZ$. For $x\in (-2,2)$, we find that as
$x_0\to\infty$
\begin{align*}
  \vp(x)\to&
        \lefteqn{.1}{\left(\frac 1 2+ \lambda(x)\right)\text{sgn}(x-x_0)
        \left\{1-\K\cos(k_0(x-x_0)) \right\}} \\
           &  -\left(\frac 1 2 -\lambda(x)\right)\text{sgn}(x+x_0)
             \left\{1-\K\cos(k_0(x+x_0))\right\} \\
  =&-\left(\frac 1 2+ \lambda(x)\right)\left\{1-\K\cos(k_0(x-x_0)) \right\} \\
           &-\left(\frac 1 2 -\lambda(x)\right)
             \left\{1-\K\cos(k_0(x+x_0))\right\} \\
  =& -1+ \K \cdot 
     \Biggl\{ %\left\{
     \left(\frac 1 2 +\lambda(x)\right)\cos(k_0(x-x_0)) \Biggr. \\
           & \qquad{} \Biggl.
             +\left(\frac 1 2 -\lambda(x)\right)\cos(k_0(x+x_0))\Biggr\} \\ %\right\} \\
           & =-1+ \K \cdot \left\{\cos(k_0x) \cos(k_0x_0) 
           + 2\lambda(x)\sin(k_0x) \sin(k_0x_0)\right\} \\
           & =-1+ \K \cdot \cos(k_0x)\cos(k_0x_0)=:\vp^\infty(x),
\end{align*}
as $\sin(k_0x_0)=0$ and $\cos(k_0x_0)=1$ is independent of $x_0\in4\ZZ$. 

On $(-2,2)$, this limit function $\vp^\infty$ solves $L\vp^\infty-\alpha\text{sgn}(\vp^\infty)=0$, since
$\cos(k_0x_0)=1$ and $\K=\KK=1-\frac{2-k_0}{c^2k_0^2-k_0}<1$ gives
\begin{equation*}
  \vp^\infty(x)=-1+ \K \cdot \cos(k_0x)\cos(k_0x_0)<0
\end{equation*} 
for all $x\in(-2,2)$. Hence
\begin{equation*}
  L\vp^\infty-\alpha\text{sgn}(\vp^\infty)
  =\K\cos(k_0x_0)L \cos(k_0\cdot)=0.
\end{equation*}
As a consequence, $\norm{L\vp-\alpha\text{sgn}(\vp)}_{L^2(-2,2)}\to 0$ as $x_0\in 4\ZZ$ tends to $\infty$.

The same argument works for $x_0\to \infty$ with $x_0\in 2\ZZ\backslash 4\ZZ$, but this time $\cos(k_0x_0)=-1$.
\end{proof}

Let us now turn to the even function $\ue$.  For example, one can choose $\ue$ to agree with $\sgn(x)\sin(k_0x)$
outside a fixed bounded interval.  The essential property used is that such a function will satisfy the condition in
Proposition~\ref{eq: orthogonality} in Appendix~\ref{sec:Appendix}.

For any choice of the parameter $\be\in \RR$ and any $\widetilde r\in H^2_\e(\RR)$, we can choose the remaining
parameter $\tilde \gamma$ to ensure that $u$ of~\eqref{eq: r tilde} inherits the two zeros $\pm x_0$ from $\vp$. That
is, we set
\begin{equation*}
  \tilde \gamma := \left\{\widetilde r(x_0)-\be\ue(x_0)\right\}\cos(k_0x_0)^{-1},
\end{equation*}
where we note that $\cos(k_0x_0)=\pm 1$ for $x_0\in 2\ZZ$.

To motivate the definition of $\tilde r$, let us assume for the moment that $\pm x_0$ are the only zeros of $u$. In
other words, let us assume for now that the sign condition
\begin{equation}
  \label{eq:two-sign}
  \text{sgn}\left(\vp+\be \ue+\tilde \gamma\cos(k_0\cdot)-\widetilde r\right) =\text{sgn}(\vp)
\end{equation}
holds. In analogy to~\eqref{eq:eqn-split} as an equation for the remainder $r$ in Section~\ref{sec:Proof-Theor}, we now
consider the equation
\begin{equation}
  \label{eq:eqn-split-two}
  L\widetilde r=\be L \ue+L\vp-\alpha\text{sgn}(\vp)
\end{equation}
for $\widetilde r\in H^2_\e(\RR)$, where the subscript $\e$ stands for even functions. Note that
if~\eqref{eq:eqn-split-two} has a solution $\tilde r$, then the function $u$, with the decomposition~\eqref{eq: r
  tilde} will be a solution to~\eqref{eq:introeq} provided the sign condition~\eqref{eq:two-sign} holds.

The solvability of~\eqref{eq:eqn-split-two} is addressed in the following lemma.
\begin{lemma}
  \label{lem:r tilde solv}
  Define
  \begin{equation*}
    \be := \frac{1}{2(c^2k_0-1)}\int_\RR        \left[-L\vp+\alpha\text{sgn}(\vp) \right]\cos(k_0\cdot)dx.
  \end{equation*}
  Then equation~\eqref{eq:eqn-split-two} has an even solution $\widetilde r\in H^2_e(\RR)$.  In particular, we have the
  estimate
  \begin{equation*}
    \norm{\widetilde r}_{H^2(\RR)}\le
    C\left(\abs\be+\norm{(1+x^2)\left(L\vp-\alpha{\rm{sgn}}(\vp)\right)}_{L^2(\RR)}\right).
  \end{equation*}
\end{lemma}

\begin{proof}
By the choice of $\be$ and Proposition~\ref{eq: orthogonality},
\begin{equation*} 
  \int_\RR\Big(\be L \ue+L\vp-\alpha\text{sgn}(\vp)\Big)\cos(k_0\cdot)dx=0.
\end{equation*}
The expression $L^{-1}Q$ given by Proposition~\ref{prop: L inverse bis} in Appendix~\ref{sec:Appendix} can be applied
to the right hand side of~\eqref{eq:eqn-split-two},
\begin{equation*}
  Q :=\be L \ue+L\vp-\alpha\text{sgn}(\vp),
\end{equation*}
because $(1+x^2)Q\in L^2(\RR)$ and $\int_\RR Q(x)\sin(k_0 x)dx=\int_\RR Q(x)\cos(k_0 x)dx=0$.  Hence
\begin{equation*}
  \widetilde r :=L^{-1}\left(\be L \ue+L\vp-\alpha\text{sgn}(\vp)\right)
\end{equation*}
is well-defined. It is immediate that $\tilde r$ is even. 
\end{proof}

\begin{theorem}
  \label{theo:two-trans}
  Under the assumptions of Theorem~\ref{theo:KZ} (in particular, for a piecewise quadratic on-site potential,
  $\psi'(x)=\sgn(x)$), there exists a family of even solutions
  \begin{equation*}
    u=\vp+\be \ue+\tilde \gamma\cos(k_0\cdot)-\widetilde   r
  \end{equation*} 
  to~\eqref{eq:introeq}, parametrised by the choice of sufficiently large $x_0\in2\ZZ$ in Lemma~\ref{lem:vp}.

  Each of these solutions making two transitions between the wells of the on-site potential, located at $-x_0$ and
  $+x_0$, so that they remain in one well only on a large but finite interval $(-x_0,x_0)$.
\end{theorem}

\begin{proof}
Lemma~\ref{lem:vp} provides $\vp$. Further, $\ue$ is as discussed above.  In addition, Lemma~\ref{lem:r tilde solv}
defines $\be$ and $\widetilde r$.

As
\begin{equation*}  
  \be = \frac{1}{2(c^2k_0-1)}\int_\RR
  \left[-L\vp+\alpha\text{sgn}(\vp) \right]\cos(k_0\cdot)dx,
\end{equation*}
we obtain by estimate~\eqref{eq:vp-norm}
\begin{align*}
  \abs\be &\le C \norm{(1+x^2)\left(L\vp-\alpha\text{sgn}(\vp)\right)}_{L^2(\RR)}\cdot
            \norm{\frac{\cos(k_0x)}{1+x^2}}_{L^2(\RR)}
            \to0
\end{align*}
for a sequence of points $x_0\in2\ZZ$ with $x_0 \to \infty$. 

%Further, $\widetilde r$, which can be made arbitrarily small in $H^2_\e(\RR)$ if
%$x_0$ is chosen sufficiently large. 

It remains to verify the sign condition~\eqref{eq:two-sign} for $u$, i.e., to show that $\pm x_0$ are the only roots of 
\begin{equation*}
  u=\vp+\be \ue+\tilde \gamma\cos(k_0\cdot)-\widetilde   r.
\end{equation*}
Recall that the choice 
\begin{equation*}
  \tilde \gamma = \left\{\widetilde r(x_0)-\be\ue(x_0)\right\}\cos(k_0x_0)^{-1}
\end{equation*}
was made so that $u$ vanishes at $\pm x_0$. The bounded embedding $H^2(\RR)\subset L^\infty(\RR)$ and Lemma~\ref{lem:r
  tilde solv} show that $\widetilde r(x_0)$ is small.  Moreover, smallness of $\be$ and $\widetilde r(x_0)$ imply that
$\tilde \gamma$ is small itself.

As $\vp$ changes sign at precisely $\pm x_0$, we now use that the derivative $\vp'(\pm x_0)$ is bounded below
independently of large $x_0$.  Thus, pointwise smallness of all additional terms
$\be \ue+\tilde \gamma\cos(k_0\cdot)-\widetilde r$ establishes the sign condition for all sufficiently large $x_0$.
\end{proof}

\appendix 
\section{Appendix}
\label{sec:Appendix}

We state a useful generalisation of Proposition~\ref{prop: L inverse}, by considering functions $Q$ which are not
necessarily odd.
\begin{proposition}
  \label{prop: L inverse bis}
  If $Q\in L^2(\RR)$ satisfies
  \begin{equation*}
    (1+x^2)Q\in L^2(\RR)~\text{ and }~
    \int_\RR Q(x)\sin(k_0 x)dx=
    \int_\RR Q(x)\cos(k_0 x)dx=0,
  \end{equation*}
  then, for all $c$ near enough to $1$, there exists a unique function $r\in H^2(\RR)$ such that
  $Lr=c^2r''-\Delta_D r+\alpha r=Q$.  Moreover,
  \begin{equation*}
    \norm{r}_{H^2(\RR)}:=\norm{(1+k^2)\widehat r}_{L^2(\RR)}
    \leq \left(C_1+((4+\alpha)C_1+1)/c^2\right)\norm{(1+x^2)Q}_{L^2(\RR)}~,
  \end{equation*}
  where the constant $C_1>0$ is as in Proposition~\ref{prop: L inverse}.
\end{proposition}

\begin{proof}
  When $Q$ is even, the proof is the same as the one of Proposition~\ref{prop: L inverse}, except that then
  $\widehat Q$, $\widehat r$ and $r$ are even and real-valued.  In general, we write $Q=Q_{\o}+Q_{\e}$, where
\begin{equation*}
  Q_{\o}(x)=\frac 1 2 (Q(x)-Q(-x))
  ~~\text{and }~~Q_{\e}(x)=\frac 1 2 (Q(x)+Q(-x))
\end{equation*}
are odd respectively even. We set
\begin{equation*}
  \widehat r_{\o}(k):=\frac{\widehat Q_{\o}(k)}{D(k)}
  ~~\text{and }~~
  \widehat r_{\e}(k):=\frac{\widehat Q_{\e}(k)}{D(k)}~,~k\in \RR,
\end{equation*}
which are odd respectively even as well. Then $r:=r_{\o}+r_{\e}$ satisfies
\begin{equation*}
  \widehat r(k)=\frac{\widehat Q(k)}{D(k)}~,~k\in \RR.
\end{equation*}
As
\begin{equation*}
  \int_{\RR}(1+k^2)^2 \widehat r_{\o} (k)\cdot\overline{\widehat  r_{\e} (k)}dk
  =\int_{\RR} (1+x^2)^2Q_{\o}(x)Q_{\e}(x)dx=0,
\end{equation*}
we obtain
\begin{multline*}
  \norm{r}^2_{H^2(\RR)}:=\norm{(1+k^2)\widehat r}^2_{L^2(\RR)}
  =\norm{(1+k^2)\widehat r_{\o}}^2_{L^2(\RR)}
  +\norm{(1+k^2)\widehat r_{\e}}^2_{L^2(\RR)}
  \\ \leq \left(C_1+((4+\alpha)C_1+1)/c^2\right)^2\left(\norm{(1+x^2)Q_{\o}}^2_{L^2(\RR)}
  + \norm{(1+x^2)Q_{\e}}^2_{L^2(\RR)}\right)
  \\= \left(C_1+((4+\alpha)C_1+1)/c^2\right)^2\norm{(1+x^2)Q}^2_{L^2(\RR)}~.
\end{multline*}
\end{proof}

The following proposition establishes orthogonality relations and estimates for the Fourier mode associated with $k_0$
for $L$ applied to even and odd functions. The estimate~\eqref{eq:uo ortho} is used just after the compactness proof
(Lemma~\ref{lem:Gamma cpt}).
\begin{proposition}
  \label{eq: orthogonality} 
  Consider the odd function $\uodd\in H^{2}_\loc(\RR)$ satisfying~\eqref{eq:u odd}. In addition, let
  $u_{\e}\in H^2_\loc(\RR)$ be an even function such that
  \begin{equation*}
    (1+x^2)\frac{\mathrm{d}^l}{\mathrm{d}x^l}( u_{\e}(x)-\mathrm{sgn}(x)\sin(k_0x))
    \in L^2(\RR\backslash[-1,1])
  \end{equation*}
  for $l=0,1,2$, analogously to~\eqref{eq:u odd}. If $c>k_0^{-1/2}$, then
  \begin{align}
    \int_{\RR}\sin(k_0\cdot )(c^2\uodd''-\Delta_D \uodd+\alpha \uodd)dx &=-2c^2k_0+2<0, \label{eq:uo ortho}\\
    \int_{\RR}\cos(k_0\cdot )(c^2u_{\e}''-\Delta_D u_{\e}+\alpha u_{\e})dx &=2c^2k_0-2>0, \notag \\
    \int_{\RR}\cos(k_0\cdot ) (c^2\uodd''-\Delta_D \uodd+\alpha \uodd)dx&=0 \notag \\
  \intertext{and}
    \int_{\RR}\sin(k_0\cdot ) (c^2u_{\e}''-\Delta_D u_{\e}+\alpha u_{\e})dx&=0. \notag
  \end{align}
\end{proposition}

\begin{proof}
  The two last integrals vanish because the integrands are odd functions of $x$.  For the first integral, two
  integrations by parts and the identity $L\sin(k_0\cdot)=0$ give
\begin{align*}
  &\lim_{z\rightarrow \infty}\int_{-z}^{z}\sin(k_0\cdot )
    \left(c^2\uodd''-\Delta_D \uodd+\alpha \uodd\right)dx\\
  &=\lim_{z\rightarrow \infty}\int_{-z}^{z}\left[
    c^2\frac{d^2}{dx^2}\sin(k_0\cdot )-\Delta_D\sin(k_0\cdot)
    +\alpha \sin(k_0\cdot)\right]\uodd\, dx \\
  &\qquad{}+\lim_{z\rightarrow \infty}c^2
    \left[\sin(k_0 z)u'_0(z)-k_0\cos(k_0 z)\uodd(z) \right.\\
  &\qquad\qquad
    {}\left.- \sin(-k_0 z)u'_0(-z)+k_0\cos(-k_0 z)\uodd(-z)\right] \\
  &\qquad{}-\lim_{z\rightarrow \infty}\left(\int_{-z+1}^{z+1}-\int_{-z}^{z}\right)\sin(k_0(x-1))\uodd(x)dx\\
  &\qquad{}-\lim_{z\rightarrow \infty}\left(\int_{-z-1}^{z-1}-\int_{-z}^{z}\right)\sin(k_0(x+1))\uodd(x)dx \\
  &\stackrel{\eqref{eq:u odd}}=
     \lim_{z\rightarrow \infty}c^2
     \left(-k_0\sin^2(k_0 z)-k_0\cos^2(k_0 z)
     -k_0 \sin^2(-k_0 z)-k_0\cos^2(-k_0 z)\right)
  \\&\qquad{}-\lim_{z\rightarrow \infty}\int_{z}^{z+1}\sin(k_0(x-1))\cos(k_0x)dx
       -\lim_{z\rightarrow \infty}\int_{-z}^{-z+1}\sin(k_0(x-1))\cos(k_0x)dx
  \\&\qquad{}+\lim_{z\rightarrow \infty}\int_{-z-1}^{-z}\sin(k_0(x+1))\cos(k_0x)dx
       +\lim_{z\rightarrow \infty}\int_{z-1}^{z}\sin(k_0(x+1))\cos(k_0x)dx
  \\&=-2c^2k_0
       +\lim_{z\rightarrow \infty}\int_{z-1}^{z+1}\cos^2(k_0x)dx
       +\lim_{z\rightarrow \infty}\int_{-z-1}^{-z+1}\cos^2(k_0x)dx
  \\&=-2c^2k_0+2<0.
\end{align*}
Analogously,
\begin{align*}
  &\int_{\RR}\cos(k_0\cdot )(c^2u_{\e}''-\Delta_D u_{\e}+\alpha u_{\e})dx
    =\int_{\RR}\sin(k_0\cdot+ k_0)
    (c^2u_{\e}''-\Delta_D u_{\e}+\alpha u_{\e})dx\\
  &=\lim_{z\rightarrow \infty}c^2
    \left(-k_0\sin(k_0 z+k_0)\sin(k_0 z-k_0)-k_0\cos(k_0z+k_0)\cos(k_0 z-k_0) \right.\\
  &~~~~~~~
       \left.-k_0 \sin(-k_0z+k_0)\sin(-k_0 z-k_0)-k_0\cos(-k_0 z+k_0)\cos(-k_0 z-k_0)\right)\\
  &\qquad{}-\lim_{z\rightarrow \infty}\int_{z}^{z+1}\sin(k_0(x-1)+k_0)\cos(k_0x-k_0)dx\\
  &\qquad{}-\lim_{z\rightarrow \infty}\int_{-z}^{-z+1}\sin(k_0(x-1)+k_0)\cos(k_0x-k_0)dx\\
  &\qquad{}+\lim_{z\rightarrow \infty}\int_{-z-1}^{-z}\sin(k_0(x+1)+k_0)\cos(k_0x-k_0)dx\\
  &\qquad{}+\lim_{z\rightarrow \infty}\int_{z-1}^{z}\sin(k_0(x+1)+k_0)\cos(k_0x-k_0)dx\\
  &=2c^2k_0-2>0.
\end{align*}
\end{proof}

\paragraph{Acknowledgement} This work was initiated at the workshop ``Solitons, Vortices, Minimal Surfaces and their
Dynamics'' at the Mittag Leffler Institute in 2013. JZ greatfully acknowledges funding by the EPSRC, EP/K027743/1. 

%\bibliographystyle{plain}
%% Syntax for ln: ln -s ~/research/biblio/johannes/jz.bib jz.bib
%\bibliography{jz}

\def\cprime{$'$} \def\cprime{$'$} \def\cprime{$'$}
  \def\polhk#1{\setbox0=\hbox{#1}{\ooalign{\hidewidth
  \lower1.5ex\hbox{`}\hidewidth\crcr\unhbox0}}} \def\cprime{$'$}
  \def\cprime{$'$}

\end{document}